\documentclass[12pt]{amsart}
\usepackage{amsfonts,amssymb,amscd,amsmath,enumerate,verbatim, color}
\usepackage[latin1]{inputenc}
\usepackage{amscd}
\usepackage{latexsym}

\usepackage{graphicx}

\usepackage{mathptmx}
\input xy
\xyoption{all}
\def\NZQ{\mathbb}               

\def\QQ{{\NZQ Q}}
\def\ZZ{{\NZQ Z}}
\def\RR{{\NZQ R}}

\def\frk{\mathfrak}               

\def\Phi{{\frk N}}
\def\ab{{\mathbf a}}
\def\bb{{\mathbf b}}

\def\eb{{\mathbf e}}
\def\tb{{\mathbf t}}

\def\ub{{\mathbf u}}

\def\xb{{\mathbf x}}
\def\yb{{\mathbf y}}
\def\zb{{\mathbf z}}

\def\opn#1#2{\def#1{\operatorname{#2}}} 
\opn\chara{char} 
\opn\length{\ell} 
\opn\pd{pd} 
\opn\rk{rk}
\opn\projdim{proj\,dim} 
\opn\injdim{inj\,dim} 
\opn\rank{rank}
\opn\depth{depth} 
\opn\grade{grade} 
\opn\height{height}
\opn\embdim{emb\,dim} 
\opn\codim{codim}

\opn\Tr{Tr} 
\opn\bigrank{big\,rank}
\opn\superheight{superheight}
\opn\lcm{lcm}
\opn\trdeg{tr\,deg}
\opn\reg{reg} 
\opn\lreg{lreg} 
\opn\ini{in} 
\opn\lpd{lpd}
\opn\size{size}
\opn\mult{mult}
\opn\dist{dist}
\opn\cone{cone}
\opn\lex{lex}
\opn\rev{rev}
\opn\div{div} \opn\Div{Div} \opn\cl{cl} \opn\Cl{Cl}
\opn\Spec{Spec} \opn\Supp{Supp} \opn\supp{supp} \opn\Sing{Sing}
\opn\Ass{Ass} \opn\Min{Min}
\opn\Ann{Ann} \opn\Rad{Rad} \opn\Soc{Soc}
\opn\Syz{Syz} \opn\Im{Im} \opn\Ker{Ker} \opn\Coker{Coker}
\opn\Am{Am} \opn\Hom{Hom} \opn\Tor{Tor} \opn\Ext{Ext}
\opn\End{End} \opn\Aut{Aut} \opn\id{id} \opn\ini{in}

\opn\nat{nat}
\opn\pff{pf}
\opn\Pf{Pf} \opn\GL{GL} \opn\SL{SL} \opn\mod{mod} \opn\ord{ord}
\opn\Gin{Gin}
\opn\Hilb{Hilb}\opn\adeg{adeg}\opn\std{std}\opn\ip{infpt}
\opn\Pol{Pol}
\opn\sat{sat}
\opn\Var{Var}
\opn\Gen{Gen}

\opn\aff{aff} \opn\con{conv} \opn\relint{relint} \opn\st{st}
\opn\lk{lk} \opn\cn{cn} \opn\core{core} \opn\vol{vol}
\opn\link{link} \opn\star{star}
\opn\gr{gr}

\def\Ac{{\mathcal A}}

\def\Hc{{\mathcal H}}

\def\Ic{{\mathcal I}}
\def\Jc{{\mathcal J}}
\def\Gc{{\mathcal G}}
\def\Fc{{\mathcal F}}
\def\Oc{{\mathcal O}}
\def\Pc{{\mathcal P}}
\def\Qc{{\mathcal Q}}

\def\Cc{{\mathcal C}}

\def\Vol{{\textnormal{Vol}}}
\def\vol{{\textnormal{vol}}}
\def\pot#1#2{#1[\kern-0.28ex[#2]\kern-0.28ex]}

\opn\dirlim{\underrightarrow{\lim}}
\opn\inivlim{\underleftarrow{\lim}}
\def\Implies{\ifmmode\Longrightarrow \else
	\unskip${}\Longrightarrow{}$\ignorespaces\fi}
\def\implies{\ifmmode\Rightarrow \else
	\unskip${}\Rightarrow{}$\ignorespaces\fi}
\def\iff{\ifmmode\Longleftrightarrow \else
	\unskip${}\Longleftrightarrow{}$\ignorespaces\fi}

\let\:=\colon
\newtheorem{Theorem}{Theorem}[section]
\newtheorem{Lemma}[Theorem]{Lemma}
\newtheorem{Corollary}[Theorem]{Corollary}
\newtheorem{Proposition}[Theorem]{Proposition}

\theoremstyle{definition}

\newtheorem{Example}[Theorem]{Example}
\newtheorem{Examples}[Theorem]{Examples}

\newtheorem{Question}[Theorem]{Question}

\let\epsilon\varepsilon
\let\phi=\varphi
\let\kappa=\varkappa
\opn\dis{dis}
\opn\height{height}
\opn\dist{dist}
\def\pnt{{\raise0.5mm\hbox{\large\bf.}}}

\opn\Lex{Lex}
\opn\conv{conv}

\begin{document}
\title{Integer decomposition property for Cayley sums of\\
 order and stable set polytopes}
\author{Takayuki Hibi, Hidefumi Ohsugi and Akiyoshi Tsuchiya}

\address{Takayuki Hibi,
Department of Pure and Applied Mathematics,
Graduate School of Information Science and Technology,
Osaka University,
Suita, Osaka 565-0871, Japan}
\email{hibi@math.sci.osaka-u.ac.jp}

\address{Akiyoshi Tsuchiya,
Graduate school of Mathematical Sciences,
University of Tokyo,
Komaba, Meguro-ku, Tokyo 153-8914, Japan} 
\email{akiyoshi@ms.u-tokyo.ac.jp}

\address{Hidefumi Ohsugi,
Department of Mathematical Sciences,
School of Science and Technology,
Kwansei Gakuin University,
Sanda, Hyogo 669-1337, Japan} 
\email{ohsugi@kwansei.ac.jp}

\subjclass[2010]{13P10, 52B20}
\keywords{Reflexive polytope, Integer decomposition property, Order polytope,
Stable set polytope, Perfect graph, Ehrhart $\delta$-polynomial, Gr\"obner basis}

\begin{abstract}
Lattice polytopes which possess the integer decomposition property 
(IDP for short)
turn up in many fields of mathematics.
It is known that if the Cayley sum of lattice polytopes 
possesses  IDP, then so does their Minkowski sum.
In this paper,  the Cayley sum of the order polytope of a finite poset and the stable set polytope of a finite simple graph
is studied.
We show that the Cayley sum of an order polytope and the stable set polytope of a perfect graph 
possesses a regular unimodular triangulation and IDP,
and hence so does their Minkowski sum.
Moreover, it turns out that, for an order polytope and the stable set polytope of a graph, the following conditions are equivalent:
(i) the Cayley sum is Gorenstein; (ii) the Minkowski sum is Gorenstein; (iii) the graph is perfect.
\end{abstract}
\maketitle
\section*{Introduction}

A {\em lattice polytope} is a convex polytope all of whose vertices have integer coordinates.
For two lattice polytopes $\Pc, \Qc \subset \RR^d$, $\Pc + \Qc = \{\xb + \yb: \xb \in \Pc,\  \yb \in \Qc \}$ is called the {\em Minkowski sum} of $\Pc$ and $\Qc$.
In \cite{Oda}, Oda posed the following fundamental question about Minkowski sums:
\begin{Question}[{\cite[Problem 1]{Oda}}]
	Let $\Pc \subset \RR^d$ and $\Qc \subset \RR^d$ be lattice polytopes.
	When does the equation
	\begin{equation}
	\label{eq:oda}
	\Pc \cap \ZZ^d + \Qc \cap \ZZ^{d} = (\Pc + \Qc) \cap \ZZ^d
	\end{equation}
	hold?
\end{Question}

In general, for two lattice polytopes $\Pc, \Qc \subset \RR^d$, the equation (\ref{eq:oda}) does not necessarily hold, even in the special case $\Pc=\Qc$.
A lattice polytope $\Pc \subset \RR^d$ possesses the {\em integer decomposition property} (IDP for short) if for any positive integer $k \geq 1$, the equation
\[
	\Pc \cap \ZZ^d + k\Pc \cap \ZZ^{d} = (k+1) \Pc \cap \ZZ^d
\]
holds, where $k\Pc =\{k\xb : \xb \in \Pc \}$ is the $k$th dilated polytope of $\Pc$.
A lattice polytope which possesses the integer decomposition property is called {\em IDP}.
IDP polytopes turn up in many fields of mathematics such as algebraic geometry and  commutative algebra, and they are particularly important in the  theory and application of integer programing \cite[\S 22.10]{integer}.
We also consider another question about Minkowski sums:
\begin{Question}
	Let $\Pc \subset \RR^d$ and $\Qc \subset \RR^d$ be IDP polytopes.
	When is $\Pc+\Qc$ IDP? 
\end{Question}
In general, for two IDP lattice polytopes $\Pc, \Qc \subset \RR^d$, $\Pc + \Qc$ is not necessarily IDP (see \cite[\S 4]{LM}).
For these two questions, to consider the integer decomposition property of Cayley sums is useful.
Given two lattice polytopes $\Pc, \Qc \subset \RR^d$, the {\em Cayley sum} of $\Pc$ and $\Qc$, denoted by $\Pc * \Qc$,
 is the convex hull of 
$
(\Pc \times \{1\})
\cup
(\Qc \times \{0\})
 \subset \RR^{d+1}$.
Then the following result is known:
\begin{Proposition}[{\cite[Theorem 0.4]{TCayley}}]
\label{ATcayley}
	Let $\Pc, \Qc \subset \RR^d$ be lattice polytopes.
		If $\Pc * \Qc$ is IDP, then
$\Pc + \Qc$ is IDP
and		$\Pc \cap \ZZ^d + \Qc \cap \ZZ^{d} = (\Pc + \Qc) \cap \ZZ^d$.
\end{Proposition}

On the other hand, we will show that if the Cayley sum possesses a regular unimodular triangulation, then so does the Minkowski sum (Theorem~\ref{thm:groebner}).
Note that $\Pc$ and $\Qc$ are facets of $\Pc * \Qc$.
Hence 
if $\Pc * \Qc$ is IDP (resp.~has a regular unimodular triangulation), then each of $\Pc$ and $\Qc$ is IDP
 (resp.~has a regular unimodular triangulation).

In this paper, we discuss when the Cayley sum of two lattice polytopes is IDP.
In particular, we focus on  $\Oc_P*\Qc_G$, where $\Oc_P$ is the order polytope of a finite partially ordered set (poset for short) $P=\{p_1,\ldots,p_d \}$ and $\Qc_G$ is the stable set polytope of a finite simple graph on $[d]:=\{1,\ldots,d\}$.
Recently, several classes of lattice polytopes arising from order polytopes and stable set polytopes have been studied (\cite{double,HTperfect,HTorderstable,harmony}).
In fact, those lattice polytopes are Gorenstein, which is an important property of lattice polytopes in commutative algebra, toric geometry and mirror symmetry.
Moreover, in \cite{HTperfect,HTorderstable,harmony}, polyhedral characterizations of perfect graphs, which are important class in classical graph theory, are given.
%
In this paper, by using an algebraic technique on Gr\"obner bases, we will prove the following theorem:

\begin{Theorem}
	\label{thm:main}
	Let $P=\{p_1,\dots,p_d\}$ be a finite poset and $G$ a finite simple graph on $[d]$.  
Then
the codegree of $\Oc_P*\Qc_G$ is $2$,
the codegree of $\Oc_P + \Qc_G$ is $1$,
and the following conditions are equivalent{\rm :}
	\begin{itemize}
		\item[(i)] The Cayley sum $\Oc_P*\Qc_G$ is Gorenstein (of index $2$){\rm ;}
		\item[(ii)] The Minkowski sum $\Oc_P + \Qc_G$ is Gorenstein (of index $1$){\rm ;}
\item[(iii)] The graph $G$ is perfect.
	\end{itemize}
Moreover, if $G$ is perfect, then we have the following{\rm :}
\begin{itemize}
\item[(a)]
Each of $\Oc_P*\Qc_G$ and $\Oc_P+\Qc_G$
possesses a regular unimodular triangulation{\rm ;}
\item[(b)]
Each of $\Oc_P*\Qc_G$ and $\Oc_P+\Qc_G$ is IDP{\rm ;}
		\item[(c)]
		$(\Oc_P \cap \ZZ^d) + (\Qc_G \cap \ZZ^d)= (\Oc_P + \Qc_G) \cap \ZZ^d$.
\end{itemize}
\end{Theorem}

This paper is organized as follows:
Section \ref{sec:pre} will provide basic materials on order polytopes, chain polytopes, stable set polytopes, and  Gorenstein polytopes, and the definition of the Ehrhart $\delta$-polynomials of lattice polytopes.
In Section \ref{sec:groebner}, we will discuss the toric ideals of Minkowski sums and Cayley sums. We will show that if the toric ideal of a Cayley sum possesses a squarefree initial ideal, then that of the Minkowski sum also possesses a squarefree initial ideal (Theorem \ref{thm:groebner}).
In Section \ref{sec:initialOS}, we will investigate the toric ideal $I_{\Oc_P*\Qc_G}$ of $\Oc_P*\Qc_G$. 
In particular, we will show that $I_{\Oc_P*\Qc_G}$ possesses a squarefree initial ideal with respect to a reverse lexicographic order if $G$ is perfect (Theorem~\ref{thm:initialOS}).
In Section \ref{sec:gorensteinOS},
we will study the Ehrhart $\delta$-polynomial of $\Oc_P*\Qc_G$ (Theorem~\ref{thm:delta}).
In Section~\ref{sec:proofmain}, we will complete the proof of Theorem \ref{thm:main}.
Finally, in Section~\ref{sec:remark}, 
we will give some remarks about $\Oc_P*\Oc_Q$ and $\Qc_G*\Qc_H$, where $P, Q$ are finite partially ordered sets with $|P|=|Q|=d$ and $G,H$ are finite simple graphs on $[d]$.

\section{Preliminary}
\label{sec:pre}
In this section, 
we summarize basic materials on order polytopes, chain polytopes, stable set polytopes, and Gorenstein polytopes. 
First, we introduce two lattice polytopes arising from a finite poset.
Let $P=\{p_1,\ldots,p_d\} $ be a finite poset.
A {\em poset ideal} of $P$ is a subset $I \subset P$
such that if $p_i \in I$ and $p_j \leq p_i$ in $P$, then $p_j \in I$.  
In particular, the empty set $\emptyset$ as well as $P$ itself is a poset ideal
of $P$.  Let $\Ic(P)$ denote the set of poset ideals of $P$.
An {\em antichain} of $P$ is a subset $A \subset P$
such that for all $p_i$ and $p_j$ belonging to $A$ with $i \neq j$ are incomparable in $P$.  
In particular, the empty set $\emptyset$ and each $1$-element subset $\{p_j\}$ are antichains of $P$.
Let $\Ac(P)$ denote the set of antichains of $P$.
Given a subset $X \subset P$,
we may associate $\rho(X) = \sum_{p_j \in X} {\eb}_j \in {\RR}^d$,
where ${\bf e}_1, \ldots, \eb_{d}$ are
the standard coordinate unit vectors of ${\RR}^d$.
In particular, $\rho(\emptyset)$ is the origin ${\bf 0}$ of $\RR^d$.
Stanley \cite{Stanley}
introduced two classes of lattice polytopes arising from finite posets, 
which are called order polytopes and chain polytopes.
The \textit{order polytope} $\Oc_P$ of $P$
is defined to be the lattice polytope which is the convex hull of 
$\{\rho(I) : I \in \Jc(P) \}.$
The \textit{chain polytope} $\Cc_P$ of $P$ is defined to be the lattice polytope which is the convex hull of $\{\rho(A) : A \in \Ac(P) \}.$
It then follows that
$\dim \Oc_P= \dim \Cc_P=d$.
It is known that each of $\Oc_P$ and $\Cc_P$ possesses
a regular unimodular triangulation and IDP.

Second, we recall a lattice polytope arising from a finite simple graph.
Let $G$ be a finite simple graph on  $[d]$ and $E(G)$ the set of edges of $G$.
(A finite graph $G$ is called simple if $G$ has no loop and no multiple edge.)
A subset $S \subset [d]$ is called {\em stable}  
if, for all $i$ and $j$ belonging to $S$ with $i \neq j$,
one has $\{i,j\} \notin E(G)$.
A {\em clique} of $G$ is a subset $C \subset [d]$ such that
for all $i$ and $j$ belonging to $C$ with $i \neq j$,
one has $\{i,j\} \in E(G)$.
Let us note that a clique of $G$ is 
a stable set of the complementary graph $\overline{G}$ of $G$.
The {\em chromatic number} of $G$ is the smallest integer 
$t \geq 1$ for which
there exist stable sets $S_{1}, \ldots, S_{t}$ of $G$ with
$[d] = S_{1} \cup \cdots \cup S_{t}$.
A finite simple graph $G$ is said to be {\em perfect} 
(\cite{sptheorem}) if, 
for any induced subgraph $H$ of $G$
including $G$ itself,
the chromatic number of $H$ is
equal to the maximal cardinality of cliques of $H$.
Now, we introduce the stable set polytopes of finite simple graphs. Let $S(G)$ denote the set of stable sets of $G$.
Then one has $\emptyset \in S(G)$ and $\{ i \} \in S(G)$
for each $i \in [d]$. 
Given a subset $X \subset [d]$, we associate $\rho(X) = \sum_{j \in X} {\eb}_j \in {\RR}^d$.
The {\em stable set polytope} $\Qc_G$ of $G$ is defined to be the lattice polytope which is the convex hull of 
$\{\rho(S): S \in S(G) \}$.
Then one has $\dim \Qc_G=d$.
Moreover, it is known that every chain polytope is the stable set polytope of a perfect graph.
If $G$ is a perfect graph, then $\Qc_G$ possesses
a regular unimodular triangulation and IDP.
However, if $G$ is not perfect, then $\Qc_G$
is not necessarily IDP 
(see \cite{stablesetnormal}).

Next, we recall what Gorenstein polytopes are.
A lattice polytope $\Pc \subset \RR^d$ of dimension $d$ is 
called {\em reflexive} 
if the origin of $\RR^d$ belongs to the interior of $\Pc$ and 
if the dual polytope 
$
\Pc^{\vee} = \{ {\bf x} \in \RR^{d} \, : \, \langle {\bf x}, {\bf y} \rangle \le 1,
\, \forall {\bf y} \in \Pc \}
$
is again a lattice polytope.  Here $\langle {\bf x}, {\bf y} \rangle$
is the canonical inner product of $\RR^d$.
It is known that reflexive polytopes correspond to Gorenstein toric Fano varieties, and they are related to
mirror symmetry (see, e.g., \cite{mirror,Cox}).
The existence of a unique interior lattice point implies that in each dimension there exist only finitely many reflexive polytopes 
up to unimodular equivalence (\cite{Lag}), 
and all of them are known up to dimension $4$ (\cite{Kre}).
A lattice polytope $\Pc \subset \RR^d$ of dimension $d$ is called {\em Gorenstein} of index $r$ if $r\Pc$ is unimodularly equivalent to a reflexive polytope.
Gorenstein polytopes are of interest in combinatorial commutative algebra, mirror symmetry, and tropical geometry (we refer to \cite{BJ,BN08,Joswig}). 

We now introduce an important invariant of a lattice polytope.
Let, in general, $\Pc \subset \RR^d$ be a lattice polytope of dimension $d$.
The {\em Ehrhart $\delta$-polynomial} of $\Pc$ is the polynomial
\[
\delta(\Pc, \lambda) = 
(1 - \lambda)^{d+1} \left[ \,
1 + \sum_{n=1}^{\infty} \mid n\Pc \cap \ZZ^d \mid \, \lambda^n \, \right]
\] 
in $\lambda$.
Then the degree of $\delta(\Pc, \lambda)$ is at most $d$,
and each coefficient of $\delta(\Pc, \lambda)$ is a nonnegative integer (\cite{RS_OHGCM}).
In addition $\delta(\Pc, 1)$ coincides with the \textit{normalized volume} of $\Pc$, denoted by $\Vol(\Pc)$.
Moreover, if $\delta(\Pc,\lambda)=\delta_0+\delta_1 \lambda+\cdots+\delta_s \lambda^s$ with $\delta_s \neq 0$, then $d+1-s$ is called the \textit{codegree}
of $\Pc$.
The codegree $d+1-s$ is the smallest positive integer $\ell$ such that $\ell \Pc$ has an interior lattice point. 
Hence one can compute the codegree of $\Pc$ by finding a positive integer $\ell$ such that
$(\ell -1) \Pc$ has no interior lattice points and $\ell \Pc$ has an interior lattice point.
It is known that 
$\Pc$ is Gorenstein of index $r$ if and only if the codegree of $\Pc$ is $r$ and $\delta_i=\delta_{s-i}$ for $i=0,\ldots,s$. 
Refer the reader to \cite{BR15} and \cite[Part II]{HibiRedBook} for the detailed information 
about Ehrhart $\delta$-polynomials. 

Finally, we introduce important facts on the toric ideal $I_\Pc$ of a lattice polytope $\Pc$.
It is known that there exists a one-to-one correspondence between 
regular triangulations of $\Pc$ and the radical of the initial ideal of $I_\Pc$.
Moreover, the regular triangulation is unimodular if and only if the corresponding initial ideal 
is generated by squarefree monomials.
See, e.g., \cite{dojo, binomialideals, Stu} for basic facts on toric ideals together with Gr\"{o}bner bases,
and their application to lattice polytopes.

\section{Gr\"{o}bner bases of toric ideals of
$\Pc_1 * \cdots * \Pc_m$ and $\Pc_1 + \cdots + \Pc_m$}
\label{sec:groebner}

Let $K$ be a field and let $K[\tb^{\pm 1}, s]=K[t_1^{\pm 1}, \ldots, t_d^{\pm 1},s]$ be the 
 ring in $d +1$ variables over $K$.
If $\ab = (a_1,\ldots,a_d)\in \ZZ^d$,
then $\tb^\ab s$ is the Laurent monomial 
$t_1^{a_1}\cdots t_d^{a_d} s \in K[\tb^{\pm 1},s]$.
Let $\Pc \subset \RR^d$ be a lattice polytope with $\Pc \cap \ZZ^d= \{\ab_1,\ldots,\ab_n\}$.
Then, the {\em toric ring} of $\Pc$ is the subalgebra $K[\Pc]$ of $K[\tb^{\pm 1},s]$ generated by ${\tb^{\ab_1} s,\ldots, \tb^{\ab_n} s}$ over $K$. 
Let $K[\xb] = K[x_1,\ldots,x_n]$ denote the polynomial ring in $n$ variables over $K$. 
The {\em toric ideal} $I_\Pc$ of $\Pc$ is the kernel of the surjective homomorphism $\pi:K[\xb] \rightarrow K[\Pc]$ defined by $\pi(x_i)= \tb^{\ab_i} s$ for $1 \le i \le n$. It is known that $I_\Pc$ is generated by homogeneous binomials. 
One of the significance of Gr\"{o}bner bases of toric ideals is due to
the following fact.

\begin{Proposition}
\label{initialcomplex}
Let $\Pc \subset \RR^d$ be a lattice polytope such that 
$\sum_{\ab \in \Pc \cap \ZZ^d} \ZZ (\ab, 1) = \ZZ^{d+1}$.
If the toric ideal $I_{\Pc}$ possesses a squarefree initial ideal, 
then $\Pc$ has a regular unimodular triangulation and IDP.
\end{Proposition}

We recall the definition of the Cayley sum of several lattice polytopes.
Given lattice polytopes $\Pc_1,\ldots,\Pc_m \subset \RR^d$ ($m \ge 2$),
the {\em Cayley sum} $\Pc_1*\cdots*\Pc_m \subset \RR^{d+m-1}$ of $\Pc_1,\ldots,\Pc_m$ is the convex hull of 
$$(\Pc_1 \times \{\eb_1\}) \cup \dots \cup (\Pc_{m-1} \times \{\eb_{m-1}\}) \cup (\Pc_m \times \{ {\bf 0}\})\subset \RR^{d+m-1}.
$$
Now, we consider Gr\"{o}bner bases of toric ideals of $\Pc_1*\cdots*\Pc_m$ and $\Pc_1+\cdots+\Pc_m$.

\begin{Theorem}
	\label{thm:groebner}
	Let $\Pc_1,\ldots,\Pc_m \subset \RR^d$ be lattice polytopes
of dimension $d$.
Suppose that each $\Pc_i$ satisfies
$\sum_{\ab \in \Pc_i \cap \ZZ^d} \ZZ (\ab, 1) = \ZZ^{d+1}$.
If the toric ideal $I_{\Pc_1 * \cdots * \Pc_m}$ possesses a squarefree initial ideal of degree at most $\ell$, then $I_{\Pc_1 + \cdots + \Pc_m}$ possesses a squarefree initial ideal of degree at most $\ell$,
and both $\Pc_1 * \cdots * \Pc_m$ and $\Pc_1 + \cdots + \Pc_m$ have a regular unimodular triangulation and IDP.
\end{Theorem}

\begin{proof}
Suppose that $I_{\Pc_1 * \cdots * \Pc_m}$ possesses a squarefree initial ideal of degree at most $\ell$.
Since each $\Pc_i$ satisfies
$\sum_{\ab \in \Pc_i \cap \ZZ^d} \ZZ (\ab, 1) = \ZZ^{d+1}$, we have
$$\sum_{\ab \in (\Pc_1* \cdots *\Pc_m)\cap \ZZ^{d+m-1}} \ZZ (\ab, 1) = \ZZ^{d+m}.$$
By Proposition \ref{initialcomplex}, the existence of a squarefree initial ideal guarantees that 
$\Pc_1* \cdots *\Pc_m$ is IDP.
By {\cite[Theorem 0.4]{TCayley}}, $\Pc_1 + \cdots + \Pc_m$ is IDP and
	$(\Pc_1 \cap \ZZ^d) +\cdots+ (\Pc_m \cap \ZZ^d)= (\Pc_1 + \cdots + \Pc_m) \cap \ZZ^d$.
We now use a result \cite[Theorem 3.5]{Shibuta}
on generalized nested configurations.
Let $K[z_1^{\pm 1}, \ldots, z_{d+m}^{\pm 1}]$
be a Laurent polynomial ring with 
$\deg (z_i) = {\bf 0} \in \QQ^m$  for  $i = 1,2,\ldots, d$
and $\deg (z_{d+j}) = \eb_j \in \QQ^m$
 for  $j = 1,2,\ldots, m$.
Let ${\mathcal B} = \bigcup_{i=1}^m {\mathcal B}_i$,
where 
${\mathcal B}_i =
\{(\ab, \ub_i) \in \ZZ^{d+m}: \ab \in \Pc_i \cap \ZZ^d\}
 $
for $i = 1,2,\ldots, m$,
and 
$\ub _1 = \eb_1 + \eb_m, \ldots, \ub _{m-1} = \eb_{m-1} + \eb_m,$
$\ub _m = \eb_m \in \RR^m$ .
Then 
$\ub _1, \ldots, \ub _m $ are linearly independent over $\QQ$ and ${\mathcal B}_i \subset 
\{\bb \in \ZZ^{d+m}: \deg (\zb^\bb) = \ub_i\}$.
Moreover, we have $ {\mathcal B}=
\{(\ab, 1):\ab \in (\Pc_1* \cdots *\Pc_m) \cap \ZZ^{d+m-1}\}$.   
Hence the toric ideal of ${\mathcal B}$ is equal to $I_{\Pc_1* \cdots *\Pc_m}$.

Let ${\mathcal A} = \{(1,\ldots,1)\} \subset \ZZ^m$.
Then the toric ideal of ${\mathcal A}$ is zero.
For ${\mathcal B}_1,\ldots,{\mathcal B}_m$ and ${\mathcal A}$,
the generalized nested configuration
${\mathcal A}[{\mathcal B}_1,\ldots,{\mathcal B}_m]$ defined in \cite[Section 3.3]{Shibuta} is 
\begin{eqnarray*}
{\mathcal A}[{\mathcal B}_1,\ldots,{\mathcal B}_m]
&=&
\{\bb_1 + \cdots + \bb_m  :  \bb_i \in {\mathcal B}_i\}\\
&=&
\{(\ab_1 +\cdots + \ab_m,  1,\ldots, 1,m) :  \ab_i \in \Pc_i \cap \ZZ^d\}\\
&=&
\{(\ab,   1,\ldots, 1,m) :  \ab \in (\Pc_1+\cdots+\Pc_m) \cap \ZZ^d \}.
\end{eqnarray*}
Hence the toric ideal of ${\mathcal A}[{\mathcal B}_1,\ldots,{\mathcal B}_m]$ is equal to 
$I_{\Pc_1+\cdots+\Pc_m}$.

Thus by \cite[Theorem 3.5]{Shibuta},
$I_{\Pc_1+\cdots+\Pc_m}$ has a squarefree initial ideal of degree $\le \ell$.
\end{proof}

\section{A Gr\"{o}bner base of $I_{\Oc_P*\Qc_G}$}
\label{sec:initialOS}

In this section,
we will show that the toric ideal $I_{\Oc_P * \Qc_G}$ possesses a squarefree initial ideal with respect to a reverse lexicographic order if $G$ is perfect.

The {\it dual} poset $(P^*,\leq_{P^*})$ of a poset $(P,\leq_{P})$ is the poset on the set $P^*= P$ such that $s \leq_{P^*} t$ 
if and only if $t \leq_P s$.
 For finite posets $(P,\leq_{P})$ and $(Q,\leq_{Q})$ with $P \cap Q =\emptyset$,
the \textit{ordinal sum} of $P$ and $Q$ is the poset $(P \oplus Q,\leq_{P \oplus Q})$ 
on $P \oplus Q=P \cup Q$ such that $s \leq_{P \oplus Q} t$ 
if (a) $s,t \in P$ and $s \leq_{P} t$,
or (b) $s,t \in Q$ and $s \leq_{Q} t$,
or (c) $s \in P$ and $t \in Q$.
Two lattice polytopes $\Pc \subset \RR^d$ and $\Qc \subset \RR^d$ are said to be {\em unimodularly equivalent} if there is an affine map $f : \RR^d \rightarrow \RR^d$ with
$f(\ZZ^d) = \ZZ^d$ and $f(\Pc) = \Qc$. 

\begin{Lemma}
\label{changelemma}
Let $P=\{p_1,\ldots,p_d\}$ be a finite poset
and let $\Qc \subset \RR^d$ be a lattice polytope containing the origin.
Then $\Oc_P * \Qc$ is unimodularly equivalent to 
$\conv \{ \Oc_{P'} \cup (- \Qc \times \{0\})  \}$,
where $P'= \{p_{d+1}\} \oplus P^*$.
\end{Lemma}

\begin{proof}
Let $f : \RR^{d+1} \rightarrow \RR^{d+1}$ be an affine map
defined by $f (\xb) = U \xb$, where $U$ is a $(d+1) \times (d+1)$
integer matrix
$$
U=
\left(
\begin{array}{ccc|c}
& &  & 1\\
    & -E_d & & \vdots\\
 & & & 1\\
\hline
 0 & \cdots & 0 & 1
\end{array}
\right).
$$
Here, $E_d$ is an identity matrix.
It then follows that $f(\ZZ^{d+1}) = \ZZ^{d+1}$,
$f(\Oc_P \times \{1\}) = \Oc_{P^*} \times \{1\}$ and
$f(\Qc \times \{0\}) = (-\Qc) \times \{0\}$. 
Since $\{{\bf 0}\} \cup ((\Oc_{P^*} \times \{1\}) \cap \ZZ^{d+1})= \Oc_{P'} \cap \ZZ^{d+1}$
and ${\bf 0} \in (-\Qc) \times \{0\}$, we have
$f(\Oc_P * \Qc) = \conv \{ \Oc_{P'} \cup (- \Qc \times \{0\})  \}$,
as desired.
\end{proof}

Let $P=\{p_1,\ldots,p_d\}$ be a finite poset
and $G$ a perfect graph on $[d]$.
Let
$$
{\mathcal R}_{P,G} = K[\{x_I\}_{I \in {\mathcal J}(P)} \cup \{y_S\}_{\emptyset \neq S \in S (G)}\cup \{ z  \}]
$$
denote the polynomial ring over $K$ and
define the surjective ring homomorphism
$\pi: {\mathcal R}_{P,G} \rightarrow K[\Oc_P * \Qc_G]
\subset K[t_1, \ldots, t_{d+1}, s]$
by the following:
\begin{itemize}
\item
$\pi( x_I ) = \tb^{\rho(I)} t_{d+1} s$,
where $I \in {\mathcal J}(P)$;
\item
$\pi( y_S ) = \tb^{\rho(S)}  s$,
where $\emptyset \neq S \in S(G)$;
\item
$\pi(z) = s$.
\end{itemize}
Then the toric ideal $I_{\Oc_P * \Qc_G}$ is the kernel of $\pi$.
Let $<_{P,G}$ denote a reverse lexicographic order on
${\mathcal R}_{P,G}$ induced by an ordering of variables such that 
\begin{itemize}
\item
$z <_{P,G} y_S <_{P,G} x_I$ for any $S \in S(G) \setminus \{\emptyset\}$ and $I \in \Jc (P)$;
\item
$x_{I'} <_{P,G} x_{I}$ for any $I, I' \in \Jc (P)$ such that $I' \supsetneq I$;
\item
$y_{S'} <_{P,G} y_{S}$ for any $S, S' \in S(G) \setminus \{\emptyset\}$ such that $S' \subsetneq S$.
\end{itemize}
It is known \cite{ASL} that 
$$
\Gc_P = \{ x_I x_{I'} - x_{I \cap I'} x_{I \cup I'} : I, I' \in \Jc (P) ,
I \not\subset I', I \not\supset I'\}
$$
is the reduced Gr\"obner basis of $I_{\Oc_P}$ with respect to 
a reverse lexicographic order induced by $<_{P,G}$.
Here we regard $x_\emptyset$ as $z$.
The initial ideal of $I_{\Oc_P}$ is generated by squarefree quadratic monomials
$$
\{ x_I x_{I'} : I, I' \in \Jc (P) ,
I \not\subset I', I \not\supset I'\}.
$$
Let $\Gc_G$ be the reduced Gr\"obner basis of $I_{\Qc_G}$ with respect to 
a reverse lexicographic order induced by $<_{P,G}$.
It is known \cite[Example~1.3 (c)]{OHcompressed} that the initial ideal of $I_{\Qc_G}$ with respect to 
{\em any} reverse lexicographic order is squarefree if $G$ is perfect.
Let 
$$
\Gc = \{ x_I y_S - x_{I \cup \{s\}} y_{S \setminus \{s\}}: 
s \in S \setminus I,
I \cup \{s\} \in \Jc (P)
 \}
.$$

\begin{Theorem}
\label{thm:initialOS}
Let $P=\{p_1,\ldots,p_d\}$ be a finite poset
and $G$ a perfect graph on $[d]$.
Then $\Gc_P \cup \Gc_G \cup \Gc$ defined above 
is a Gr\"obner basis of $I_{\Oc_P * \Qc_G}$ with respect to a reverse lexicographic order $<_{P, G}$.
Moreover, the initial ideal of $I_{\Oc_P * \Qc_G}$ is squarefree
with respect to $<_{P, G}$.
\end{Theorem}

\begin{proof}
In order to use Lemma \ref{changelemma}, we consider the polytope
$\Pc = \conv \{ \Oc_{P'} \cup (-\Qc_{\widehat{G}}) \}$, where $P'= \{p_{d+1}\} \oplus P^*$
and $\widehat{G}$ is a graph on $[d+1]$ whose edge set is 
$E(\widehat{G}) = E(G) \cup \{\{i, d+1\} : i  \in [d] \}$.
Since $G$ is perfect, it is easy to see that $\widehat{G}$ is perfect.
Let 
$$
{\mathcal R}_{P',\widehat{G}} = K[\{x_I'\}_{\emptyset \neq I \in {\mathcal J}(P')} \cup \{y_S\}_{\emptyset \neq S \in S (\widehat{G})}\cup \{ z  \}]
$$
denote the polynomial ring over $K$ and
define the surjective ring homomorphism
$\pi': {\mathcal R}_{P',\widehat{G}} \rightarrow K[\Pc]
\subset K[t_1^{\pm 1}, \ldots, t_{d+1}^{\pm 1}, s]$
by the following:
\begin{itemize}
\item
$\pi'( x_I' ) = \tb^{\rho(I)} s$,
where $\emptyset \neq I \in {\mathcal J}(P')$;
\item
$\pi'( y_S ) = \tb^{-\rho(S)}  s$,
where $\emptyset \neq S \in S(\widehat{G})$;
\item
$\pi'(z) = s$.
\end{itemize}
Then the toric ideal $I_\Pc$ is the kernel of $\pi'$.
Let $<_{\rm rlex}$ denote a reverse lexicographic order on
${\mathcal R}_{P',\widehat{G}} $ induced by an ordering of variable such that
\begin{itemize}
\item
$z <_{\rm rlex} y_S <_{\rm rlex} x_I'$ for any $S \in S(\widehat{G}) \setminus \{\emptyset\}$ and $I \in \Jc (P') \setminus \{\emptyset\}$;
\item
$x_{I'}' <_{\rm rlex} x_{I}'$ for any $I, I' \in \Jc (P') \setminus \{\emptyset\}$ such that $I' \subsetneq I$;
\item
$y_{S'} <_{\rm rlex} y_{S}$ for any $S, S' \in S(\widehat{G}) \setminus \{\emptyset\}$ such that $S' \subsetneq S$.
\end{itemize}
In \cite[Proof of Proposition 3.1]{HTperfect}, it was shown that 
$\overline{\Gc} :=\Gc_{P'}' \cup \Gc_{\widehat{G}}' \cup \Gc'$ is a Gr\"obner basis of $I_\Pc$ with respect to $<_{\rm rlex}$, where 
\begin{eqnarray*}
\Gc_{P'}' &=& \{ x_I' x_{I'}' - x_{I \cap I'}' x_{I \cup I'}' : I, I' \in \Jc (P') \setminus \{\emptyset\},\ 
I \not\subset I', I \not\supset I'\},\\
\Gc' &=&  \{ x_I' y_S - x_{I \setminus \{s\}}' y_{S \setminus \{s\}}: 
I \in \Jc (P') \setminus \{\emptyset\}, \ S \in S(\widehat{G})\setminus \{\emptyset\}, s \in \max(I) \cap S
 \},
\end{eqnarray*}
and $\Gc_{\widehat{G}}' \subset K[\{y_S\}_{\emptyset \neq S \in S (\widehat{G})}\cup \{ z  \}]$ is the reduced Gr\"obner basis of $I_{\Qc_{\widehat{G}}}$ with respect to
a reverse lexicographic order induced by $<_{\rm rlex}$.
(Here, we regard $x_\emptyset'$ and $y_\emptyset$ as $z$.)
Moreover, the corresponding initial ideal ${\rm in}_{<_{\rm rlex}} (I_\Pc)$ is squarefree.
Since the vertex $d+1$ is adjacent to all other vertices in $\widehat{G}$,
$d+1$ belongs to $S \in S(\widehat{G})$ if and only if $S = \{d+1\}$.
It then follows that $y_{\{d+1\}}$ does not appear in $\Gc_{\widehat{G}}'$.
Moreover, we have $S \setminus \{s\} \neq \{d+1\}$ if $s \in S \in S(\widehat{G})$.
Hence $y_{\{d+1\}}$ appears only in a binomial
$x_{\{p_{d+1}\}}' y_{\{d+1\}} - z^2 \in \Gc'$ whose initial monomial is the first monomial.
By the elimination theorem \cite[Theorem 1.4.1]{dojo}, the set 
$\overline{\Gc} \setminus \{x_{\{p_{d+1}\}}' y_{\{d+1\}} - z^2\}$ is a Gr\"obner basis of $I_{\Pc'}$ with respect to $<_{\rm rlex}$, where $\Pc'=\conv \{ \Oc_{P'} \cup (- \Qc_G \times \{0\})  \}$.

Thus by Lemma \ref{changelemma}, 
 a Gr\"obner basis of $I_{\Oc_P * \Qc_G}$ with respect to $<_{P,G}$
is obtained by replacing variable
$x_I'$ with the variable $x_{P \setminus I}$
in
$\overline{\Gc} \setminus \{x_{\{p_{d+1}\}}' y_{\{d+1\}} - z^2\}$.
\end{proof}

From Theorem \ref{thm:groebner}, we immediately obtain the following corollary:
\begin{Corollary}
	\label{cor:minkowskiOS}
	Let $P=\{p_1,\ldots,p_d\}$ be a finite poset and $G$ a perfect graph on $[d]$. 
	Then $I_{\Oc_P+\Qc_G}$ possesses a squarefree initial ideal with respect to a reverse lexicographic order
and both $\Oc_P*\Qc_G$ and $\Oc_P+\Qc_G$ have a regular unimodular triangulation and IDP.
\end{Corollary}

Let $P=\{p_1,\ldots,p_d\}$ and $Q=\{q_1,\ldots,q_d\}$ be finite posets.
Let $G$ be the comparability graph of $Q$.
It is known that $G$ is perfect and the chain polytope $\Cc_Q \subset \RR^d$ of 
$Q$ is the stable set polytope $\Qc_{G} \subset \RR^d$ of $G$.
Hibi--Li \cite{hibili} proved that the reduced Gr\"obner basis $\Gc_G$ of $I_{\Qc_G}$ with respect to a reverse lexicographic order
induced by $<_{P,G}$ consists of quadratic binomials.
Thus we have the following by combining Theorems \ref{thm:groebner} and  \ref{thm:initialOS}.

\begin{Corollary}
Let $P=\{p_1,\ldots,p_d\}$ and $Q=\{q_1,\ldots,q_d\}$ be finite posets.
Then each of $I_{\Oc_P * \Cc_Q}$ and $I_{\Oc_P+\Cc_Q}$ possesses a squarefree quadratic initial ideal
with respect to a reverse lexicographic order.
\end{Corollary}

\section{The Ehrhart $\delta$-polynomial of $\Oc_P*\Qc_G$}
\label{sec:gorensteinOS}

Let $P=\{p_1,\ldots,p_d\}$ be a finite poset and  
	$G$ a perfect graph on $[d]$.
In this section, 
we will study the Ehrhart $\delta$-polynomial of $\Oc_P*\Qc_G$.
For lattice polytopes $\Pc, \Qc \subset \RR^d$, let 
$\Gamma(\Pc, \Qc)=
\conv\{
\Pc \cup (-\Qc)
\}$.
In \cite{HTorderstable}, it was shown that the polytope 
$\Gamma(\Oc_{P}, \Qc_{G})$
is a reflexive polytope with IDP
and
$$
	\delta(\Gamma(\Oc_{P}, \Qc_{G}), \lambda)=\delta(\Gamma(\Cc_{P}, \Qc_{G}),\lambda)
$$
holds.
Using this fact, we have the following.

\begin{Theorem}
	\label{thm:delta}
	Let $P=\{p_1,\ldots,p_d\}$ be a finite poset and  
	$G$ a perfect graph on $[d]$.
	Then one has
	\[
	\delta(\Oc_{P} * \Qc_{G}, \lambda)=\delta(\Gamma(\Oc_{P}, \Qc_{G}), \lambda).
	\]
	In particular, $\Oc_{P} * \Qc_{G}$ is Gorenstein of index $2$ and
	\[\Vol(\Oc_{P} * \Qc_{G})=\Vol(\Gamma(\Oc_{P}, \Qc_{G})).\] 
\end{Theorem}

\begin{proof}
By Lemma~\ref{changelemma}, $\Oc_P * \Qc_G$ is unimodularly equivalent to
$$\Pc'=\conv \{ \Oc_{P'} \cup (- \Qc_G \times \{0\})  \}
=
\conv \{ (\Oc_{P^*} \times \{1\}) \cup (- \Qc_G \times \{0\})  \},$$
where $P'=\{p_{d+1} \}\oplus P^*$.
Work with the same notation as in Proof of Theorem~\ref{thm:initialOS}.
It was proved that
$$
\widetilde{\Gc} = 
(\Gc_{P'}' \cup \Gc_{\widehat{G}}' \cup \Gc') \setminus \{x_{\{p_{d+1}\}}' y_{\{d+1\}} - z^2\}$$ is a Gr\"obner basis of $I_{\Pc'}$ with respect to $<_{\rm rlex}$, where 
\begin{eqnarray*}
\Gc_{P'}' 
&=& \{ x_I' x_{I'}' - x_{I \cap I'}' x_{I \cup I'}' : I, I' \in \Jc (P') \setminus \{\emptyset\},\ 
I \not\subset I', I \not\supset I'\},\\
\Gc' &=&  \{ x_I' y_S - x_{I \setminus \{s\}}' y_{S \setminus \{s\}}: 
I \in \Jc (P') \setminus \{\emptyset\}, \ S \in S(\widehat{G})\setminus \{\emptyset\}, s \in \max(I) \cap S
 \},
\end{eqnarray*}
where $x_\emptyset' = y_\emptyset = z$
and $\Gc_{\widehat{G}}' \subset K[\{y_S\}_{\emptyset \neq S \in S (\widehat{G})}\cup \{ z  \}]$ is the reduced Gr\"obner basis of $I_{\Qc_{\widehat{G}}}$ with respect to
a reverse lexicographic order induced by $<_{\rm rlex}$.
Note that $I \in \Jc(P')$ is not empty if and only if $p_{d+1} \in I$.
Hence we have
\begin{eqnarray*}
\Gc_{P'}' 
&=& \{ x_I'' x_{I'}'' - x_{I \cap I'}'' x_{I \cup I'}''  : I, I' \in \Jc (P^*),\ 
I \not\subset I', I \not\supset I'\},\\
\Gc' &=&  \{ x_I'' y_S - x_{I \setminus \{s\}}'' y_{S \setminus \{s\}}: 
I \in \Jc (P^*), \ S \in S(G)\setminus \{\emptyset\}, s \in \max(I) \cap S
 \}\\
& & \cup \ \  \{x_{\{p_{d+1}\}}' y_{\{d+1\}} - z^2\},
\end{eqnarray*}
where 
$x''_I = x_{I\cup \{p_{d+1}\}}'$ for each $I \in \Jc(P^*)$.
Moreover, since the variable $y_{\{d+1\}}$ does not appear in $\Gc_{\widehat{G}}'$,
the set $\Gc_{\widehat{G}}' \subset K[\{y_S\}_{\emptyset \neq S \in S (G)}\cup \{ z  \}]$ is the reduced Gr\"obner basis of $I_{\Qc_{G}}$ with respect to
a reverse lexicographic order induced by $<_{\rm rlex}$.
Then one can obtain a Gr\"obner basis of $I_{\Gamma(\Oc_{P^*}, \Qc_{G})}$ by replacing the variable
$x''_\emptyset (= x_{\{p_{d+1}\}}')$ with $z$ in $\widetilde{\Gc}$
from \cite[Proof of Proposition 3.1]{HTperfect}.
Since the variable $x_{\{p_{d+1}\}}'$ does not appear in 
the initial monomial of the binomials in $\widetilde{\Gc}$, 
the initial ideals of $I_{\Pc'}$ and $I_{\Gamma(\Oc_{P^*}, \Qc_{G})}$
have the same minimal set of monomial generators.
Hence 
$$
\left.
K[\{x_I'\}_{\emptyset \neq I \in {\mathcal J}(P')} \cup \{y_S\}_{\emptyset \neq S \in S (G)}\cup \{ z  \}] \right/ {\rm in}_{<_{\rm rlex}} (I_{\Pc'})
$$
is isomorphic to 
$$
\left( \left.
 K[\{x_I''\}_{\emptyset \neq I \in {\mathcal J}(P^*)} \cup \{y_S\}_{\emptyset \neq S \in S (G)}\cup \{ z  \}] \right/ {\rm in}_{<_{\rm rlex}} (I_{\Gamma(\Oc_{P^*}, \Qc_{G})})
\right)
\left[x_{\emptyset}'' \right].
$$
In general, 
if a lattice polytope $\Pc$ is IDP, then
the Ehrhart $\delta$-polynomial of $\Pc$
coincides with the $h$-polynomial $h(K[\Pc], \lambda)$
of the toric ring $K[\Pc] \simeq K[\xb]/I_\Pc$.
Moreover, for a monomial order $<$ on $K[\xb]$,
we have $h(K[\xb]/I_\Pc, \lambda) = h(K[\xb]/{\rm in}_< (I_\Pc), \lambda) $.
Thus we have 
\begin{eqnarray*}
\delta(\Oc_{P} * \Qc_{G}, \lambda) 
&=& 
\delta(\Gamma(\Oc_{P^*}, \Qc_{G}), \lambda) \\
&=& 
\delta(\Gamma(\Cc_{P^*}, \Qc_{G}), \lambda) \\
&=& 
\delta(\Gamma(\Cc_{P}, \Qc_{G}), \lambda) \\
&=& 
\delta(\Gamma(\Oc_{P}, \Qc_{G}), \lambda)
\end{eqnarray*}
by \cite[Theorem 1.4]{HTorderstable} and  $\Cc_{P} = \Cc_{P^*}$.
Since $\Gamma(\Oc_{P}, \Qc_{G})$ is reflexive
(Gorenstein of index $1$), $\Oc_{P} * \Qc_{G}$ is Gorenstein of index $2$.
\end{proof}

A {\em linear extension} of a poset $P=\{p_1,\ldots, p_d\}$ is a permutation 
$\sigma = i_1 i_2 \cdots i_d$ of $[d]$ which satisfies
$i_a < i_b$ if $p_{i_a} < p_{i_b}$ in $P$. 
For $W \subset [d]$,
we let $(\Delta_W(P,Q),\leq_{W})$ be the ordinal sum of $P_W$ and $Q_{\overline{W}}$, where $\overline{W}=[d] \setminus W$.
From \cite[Corollary 1.5]{Tvolume}, we obtain the following corollary.

\begin{Corollary}
	Let $P=\{p_1,\ldots,p_d\}$ and $Q=\{q_1,\ldots,q_d\}$ be finite posets.
	Then we have
	\[
	\textnormal{Vol}(\Oc_P * \Cc_Q)=\sum\limits_{W \subset [d]}e(\Delta_W(P,Q)),\]
	where $e(\Delta)$ is the number of linear extensions of a finite poset $\Delta$
\end{Corollary}

\section{The proof of Theorem \ref{thm:main}}
\label{sec:proofmain}

In this section, we complete the proof of Theorem \ref{thm:main}.
Recall that an \textit{odd hole} of a graph is an induced odd cycle of length $\geq 5$ and an \textit{odd antihole} of a graph is the complementary graph of an odd hole.
In order to prove Theorem  \ref{thm:main}, we need the {\em Strong Perfect Graph Theorem}:

\begin{Lemma}[{\cite[1.2]{sptheorem}}]
	\label{strong}
	A graph is perfect if and only if it has neither odd holes nor odd antiholes as induced subgraphs.
\end{Lemma}

On the other hand, there is a characterization for a vertex of the dual polytope of a lattice polytope $\Pc \subset \RR^d$ of
	dimension $d$ containing the origin in its interior. 

\begin{Lemma}[{\cite[Corollary 35.6]{HibiRedBook}}]
	\label{facet}
	Let $\Pc \subset \RR^d$ be a lattice polytope of
	dimension $d$ containing the origin in its interior. 
	Then a point $\ab \in \RR^d$ is a vertex of $\Pc^\vee$ if 
	and only if $\Hc \cap \Pc$ is a facet of $\Pc$,
	where $\Hc$ is the hyperplane
	$\left\{ \xb \in \RR^d : \langle \ab, \xb \rangle =1 \right\}$
	in $\RR^d$.
\end{Lemma}

Now, we prove Theorem \ref{thm:main}.

\begin{proof}[Proof of Theorem \ref{thm:main}]
Since $\Oc_P * \Qc_{G} \subset \RR^{d+1}$ has no interior lattice points and 
$$
\sum_{i=1}^{d+1}
\eb_i
=
\frac{1}{d+1} \cdot {\bf 0}
+
\frac{1}{d+1} \cdot \eb_1
+
\frac{1}{d+1} \cdot \eb_2
+
\cdots
+
\frac{1}{d+1} \cdot \eb_{d+1}
+
\frac{d}{d+1} 
\sum_{i=1}^{d+1}
\eb_i
\in \ZZ^{d+1}$$
is in the interior of $2(\Oc_P * \Qc_G)$,  
the codegree of $\Oc_P * \Qc_G$ is 2.
On the other hand, $\Oc_P + \Qc_{G} \subset \RR^{d}$ has an interior lattice point
$$
\sum_{i=1}^{d}
\eb_i
=
\frac{1}{d+1} \cdot {\bf 0}
+
\frac{1}{d+1} \cdot \left(\eb_1 +\sum_{i=1}^{d} \eb_i\right)
+
\cdots
+
\frac{1}{d+1} \cdot \left(\eb_d +\sum_{i=1}^{d} \eb_i\right)
\in \ZZ^{d}$$
and hence the codegree of $\Oc_P + \Qc_G$ is 1.
The implication (iii) $\Rightarrow$ (i) was given in Theorem~\ref{thm:delta}. 
Moreover, by \cite[Theorem~2.6]{BN08} 
we have  (i) $\Leftrightarrow$ (ii).
If $G$ is perfect, then desired conditions (a), (b) and (c)
are proved
by  Proposition \ref{ATcayley},
Theorems~\ref{thm:groebner}
and
\ref{thm:initialOS}.
Thus it is enough to show (i) $\Rightarrow$ (iii).

	((i) $\Rightarrow$ (iii))
	We prove that if $G$ is not perfect,
	then $\Oc_P * \Qc_{G}$ is not  Gorenstein of index $2$.
It is enough to show that $\Pc:=2(\Oc_P * \Qc_G)-\sum_{i=1}^{d+1}
\eb_i
$ is not reflexive.
	
	First, we suppose $G$ has an odd hole $C$ of length $2\ell+1$, where $\ell \geq 2$.
	By renumbering the vertex set of $G$, we may assume that  the edge set of $C$ is $\{\{i,i+1\} : 1\leq i \leq 2\ell \} \cup \{1,2\ell+1\} $.
	Then the hyperplane $\Hc' \subset \RR^{d+1}$ defined by the equation $z_1+\cdots
+z_{2\ell+1}-(\ell+1) z_{d+1}= \ell$ is a supporting hyperplane of 
	$\Pc$.
	Let $\Fc$ be a facet of $\Pc$ with $\Hc' \cap \Pc \subset \Fc$
	and $a_1z_1+\cdots+a_{d+1}z_{d+1}=1$ with each $a_i \in \RR$ the equation of the supporting hyperplane $\Hc \subset \RR^{d+1}$ with $\Fc \subset \Hc$.
	The maximal stable sets of $C$ are
	$$S_1=\{1,3,\ldots,2\ell-1\},S_2=\{2,4,\ldots,2\ell \},\ldots,S_{2\ell+1}=\{2\ell+1,2,4,\ldots,2\ell-2\}$$
	and each $i \in [2\ell+1]$ appears $\ell$ times in the above list.
	Since for each $S_i$, we have 
$$\sum_{j \in S_i}a_j -\sum_{j \in [d+1] \setminus S_i} a_j =1,$$
it follows that $a_1=\cdots=a_{2\ell+1}$
and hence $-a_1-(a_{2\ell+2}+\cdots+a_{d+1})=1$.
	Moreover, since $a_1+\cdots+a_{d+1}=1$, one has
$(2\ell+1) a_1+a_{2\ell+2}+\cdots+a_{d+1}=1$.
	Thus we have $a_1=1/\ell \notin \ZZ$.
	This implies that $\Pc$ is not reflexive.
	
	Next, we suppose that $G$ has an odd antihole $C$ such that the length of $\overline{C}$ equals $2\ell+1$, where $\ell \geq 2$.
	Similarly, we may assume that the edge set of $\overline{C}$ is $\{\{i,i+1\} : 1\leq i \leq 2\ell \} \cup \{1,2\ell+1\}$.
	Then the hyperplane $\Hc' \subset \RR^{d+1}$ defined by the equation $z_1+\cdots+z_{2\ell+1} - (2\ell-1) z_{d+1}= 2$ is a supporting hyperplane of 
	$\Pc$.
	Let $\Fc$ be a facet of $\Pc$ with $\Hc' \cap \Pc \subset \Fc$
	and $a_1z_1+\cdots+a_{d+1}z_{d+1}=1$ with each $a_i \in \RR$ the equation of the supporting hyperplane $\Hc \subset \RR^{d+1}$ with $\Fc \subset \Hc$.
	Then since the maximal stable sets of $C$ are the edges of $\overline{C}$,
	for each edge $\{i,j\}$ of $C$, we have 
$$a_i+a_j-\sum_{k \in [d+1] \setminus \{i,j\}}a_k=1.$$
It then follows that $a_1=\cdots=a_{2\ell+1}$
and hence $-(2\ell-3)a_1-(a_{2\ell+2}+\cdots+a_{d+1})=1$.
	Moreover, since $a_1+\cdots+a_{d+1}=1$, one has
$(2\ell+1) a_1+ a_{\ell+2}+\cdots+ a_{d+1}=1$.
	Hence $a_1=1/2 \notin \ZZ$.
	This implies that $\Pc$ is not reflexive.
\end{proof}

\section{Remarks}
\label{sec:remark}

In this section,
we will give some remarks about $\Oc_P*\Oc_Q$ and $\Qc_G*\Qc_H$, where $P$ and $Q$ are finite partially ordered sets with $|P|=|Q|=d$ and $G$ and $H$ are finite simple graphs on $[d]$.
Since $\Oc_P$ and $\Oc_Q$ are lattice polytopes of 
``type ${\bf A}$'' (\cite[Section~3.1]{HPPS}), 
we have the following by \cite[Lemma~4.15]{HPPS}
together with Theorem~\ref{thm:groebner}:

\begin{Proposition}
Let $P$ and $Q$ be finite posets with $|P|=|Q|=d$.
Then $\Oc_P*\Oc_Q$ has a regular unimodular triangulation and hence so does $\Oc_P + \Oc_Q$.
In particular, 
both $\Oc_P*\Oc_Q$ and $\Oc_P + \Oc_Q$ are IDP.
\end{Proposition}

By Lemma \ref{changelemma}, the Cayley sum $\Oc_P*\Oc_Q$ 
 is unimodularly equivalent to the convex hull of
$(\Oc_{P^*} \times \{1\}) \cup (- \Oc_{Q} \times \{0\})$.
The convex hull of 
$(2\Oc_{P} \times \{1\}) \cup (- 2\Oc_{Q} \times \{-1\})$, which is unimodularly equivalent to 
$2 \cdot \conv ((\Oc_{P} \times \{1\}) \cup (- \Oc_{Q} \times \{0\}))$,
 was studied in \cite{double}.
In particular, by \cite[Corollary 2.8 and Proposition 2.22]{double}, we have the following.

\begin{Proposition}
Let $P$ and $Q$ be finite posets with $|P|=|Q|=d$.
Then $\Oc_P*\Oc_Q$ is Gorenstein of index 2
if and only if $P^*$ and $Q$ have a common linear extension.
\end{Proposition}

\begin{Example}
Let  $P=\{p_1, \ldots, p_d\}$ be a finite poset.
Then we have $\Oc_P * \Oc_P = \Oc_{P'}$,
where $P'$ is a disjoint union of $P$ and $\{p_{d+1}\}$.
It is known \cite{ASL} that,  $\Oc_{P'}$ is Gorenstein if and only if
 all maximal chains of $P'$ have the same length.
Hence $\Oc_P*\Oc_P$ is Gorenstein
if and only if $P$ is an antichain.
\end{Example}

On the other hand, there are examples
of  perfect graphs $G$ and $H$ on $[d]$
such that  
the Cayley sum $\Qc_G*\Qc_H$ is not IDP.
Recall that, for a finite poset $P$ on $[d]$, 
the comparability graph $G$ of $P$ is a perfect graph on $[d]$
and we have $\Cc_P= \Qc_G$.

\begin{Examples}
(a) Let $P$ and $Q$ be posets on $\{1,2,3,4,5\}$ defined by the Hasse diagram in Figure~\ref{poset1}.
\begin{figure}[h]
\includegraphics[width=7.5cm,pagebox=cropbox,clip]{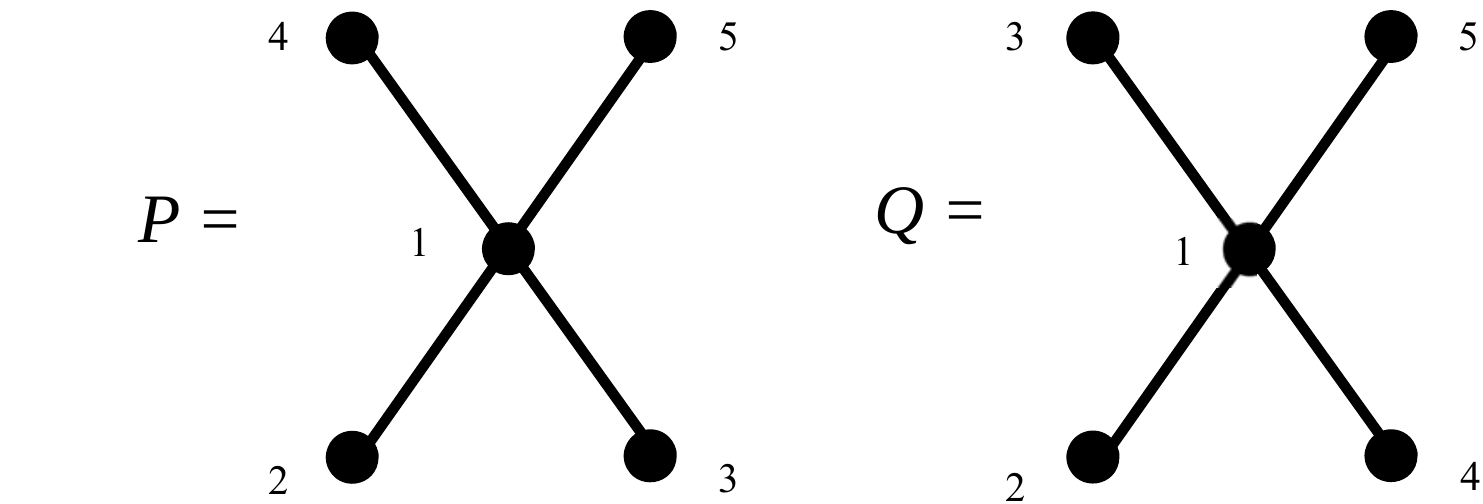}
\caption{$P$ and $Q$ such that $\Cc_P * \Cc_Q$ is not IDP.}
\label{poset1}
\end{figure}
Then $\Cc_P * \Cc_Q$ is not IDP.
On the other hand, $\Cc_P + \Cc_Q$ is IDP.

(b)
Let $P$ and $Q$ be posets on $\{1,2,3,4,5,6\}$ defined by the Hasse diagram in Figure~\ref{poset2}.
\begin{figure}[h]
\includegraphics[width=7.5cm,pagebox=cropbox,clip]{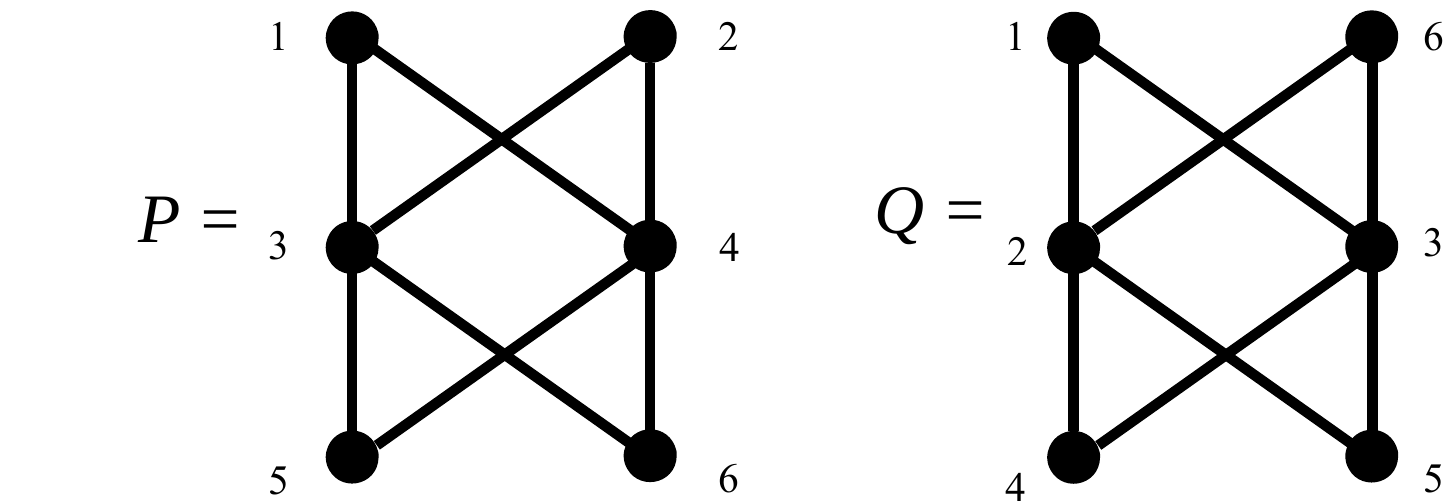}
\caption{$P$ and $Q$ such that $\Cc_P + \Cc_Q$ is not IDP.}
\label{poset2}
\end{figure}
Then neither $\Cc_P * \Cc_Q$ nor $\Cc_P + \Cc_Q$ is IDP.

(c)
Let $G$ be a perfect graph $G$ on $[d]$.
Then, we have $\Qc_G * \Qc_G = \Qc_{G'}$,
where $G'$ is a perfect graph obtained by adding an isolated vertex $d+1$ to $G$.
Thus $\Qc_G * \Qc_G$ has a regular unimodular triangulation
and hence IDP.
On the other hand, it is known \cite[Theorem 2.1 (b)]{OHspecial} that 
$ \Qc_{G'}$ is Gorenstein if and only if
all maximal cliques of $G'$ have the same cardinality.
Hence $\Qc_G * \Qc_G$ is Gorenstein if and only if $G$ is an empty graph.
\end{Examples}

\subsection*{Acknowledgment}
The authors are grateful to an anonymous referee for
his careful reading of the manuscript and for his comments.
The authors were partially supported by JSPS KAKENHI 26220701, 18H01134 and 16J01549.

\end{document}